\documentclass[12pt,a4paper,english]{smfart}
\usepackage[english]{babel}
\usepackage{ragged2e}
\usepackage{smfthm, mathabx}
\usepackage{stmaryrd}
\usepackage{amssymb}
\usepackage{amsmath}
\usepackage{amsthm}
\usepackage{amsfonts}
\usepackage{graphicx}
\usepackage{enumerate}
\usepackage{comment}
\usepackage{url}

\usepackage{tikz-cd}

\usepackage{appendix}

\usepackage{euscript,mathrsfs}
\usepackage[utf8]{inputenc}
\usepackage{longtable}
\usepackage{dsfont}

\usepackage[OT2,T1]{fontenc}

\usepackage[all,cmtip]{xy}
\usepackage{alltt}

\usepackage{calrsfs}

\xyoption{all}

\usepackage{float}

\usepackage{fullpage}

\usepackage{color}

\newtheorem*{conjectureNC}{Conjecture}

\newtheorem{theoremA}{Theorem}

\newtheorem{heuristic}{Heuristic}

\author[Oussama Hamza]{Oussama Hamza}
\address{Institute for Advanced Studies in Mathematics, Harbin Institute of Technology, Harbin, China 150001}
\email{ohamza3@uwo.ca}

\author{Donghyeok Lim}
\address{Department of Mathematics Education, Korea National University of Education, Cheongju 28173, South Korea}
\email{donghyeokklim@gmail.com}

\author{Christian Maire}
 \address{Université Marie et Louis Pasteur,  CNRS, Institut FEMTO-ST, F-25000 Besançon, France}
\email{christian.maire@univ-fcomte.fr}

\title{Massey products and unipotent extensions with restricted ramification}

\date{\today}

\dedicatory{Dedicated to J{\'a}n Min{\'a}{\v c} on his 71st birthday}

\begin{abstract}
We fix a prime~$p$ and construct new cases of pro-$p$ extensions of number fields with restricted ramification and splitting, whose Galois groups decompose as coproducts of pro-$p$ absolute Galois groups of local fields. As a consequence, these pro-$p$ extensions satisfy the strong Massey vanishing property and thus admit large unipotent quotients.
\end{abstract}

\subjclass{55S30, 11R32, 12G05, 20E06, 20F05}
\keywords{Pro-$p$ extensions with restricted ramification and splitting, Mild presentations, Massey products, Unipotent representations}
\thanks{The first and third authors gratefully acknowledge the support of the Western Academy for Advanced Research
(WAFAR) in Western University, during the year 2022/23 and the Institute for Advanced Studies in Mathematics (IASM) in Harbin Institute of Technology, during the summer 2025. The third author is grateful to WAFAR and IASM for support during visits in summers 2023 and 2025. The third author was
partially supported by the EIPHI Graduate School (contract “ANR-17-EURE-0002") and by the Bourgogne Franche-Comté Region. The second author was supported by the National Research Foundation of Korea(NRF) grants No. RS-2024-00462910. The authors thank Elyes Boughattas for his comments and interest in this paper.}

\newcommand{\Q}{\mathbb{Q}}
\newcommand{\F}{\mathbb{F}}
\newcommand{\Z}{\mathbb{Z}}
\newcommand{\NN}{\mathbb{N}}

\def\U{\mathbb{U}}

\def\P{\mathbb{P}}

\def\G{{\mathcal G}}

\def\p{{\mathfrak p}}
\def\q{{\mathfrak q}}

\def\Frat{{\rm Frat}}

\def\Cl{{\rm Cl}}

\def\O{{\mathcal O}}

\def\PP{{\mathcal P}}

\def\Uu{{\mathcal U}}

\def\AA{{\mathbb A}}

\def\rk{{\rm rank}}

\def\I{{\mathcal I}}

\def\Ff{{\mathcal F}}

\def\T{{\mathcal T}}
\def\D{{\mathcal D}}

\def\bU{{\mathbb U}}

\def\L{{\rm L}}

\def\p{{\mathfrak p}}
\def\l{{\mathfrak l}}

\def\rk{{\rm rk}}

\begin{document}

\maketitle

\section*{Introduction}

\subsection*{Context} The absolute Galois group of a field is a fundamental object in algebra and arithmetic. A natural and longstanding question is:
\smallskip

\begin{center}
\textit{Which profinite groups can be realized as absolute Galois groups of fields?}
\end{center}

\smallskip
In general, fully understanding the structure of absolute Galois groups is highly mysterious and notoriously difficult. A common approach is therefore to fix a prime $p$ and study the maximal pro-$p$ quotient, known as the pro-$p$ absolute Galois group.
For number fields, it is common to consider further quotients of the pro-$p$ absolute Galois group by restricting the ramification.
In this context, presentations of pro-$p$ groups in terms of generators and relations play a central role. A pioneering result was given by Golod and Shafarevich in the 1960s~\cite{Golod}, who used the idea of presentation to construct the first examples of infinite~$p$-class field towers. Their method was  systematically extended to the setting of pro-$p$ absolute Galois groups unramified outside finite sets~$S$ of primes, that we denote by~$G_S$~(a detailed introduction to these Galois groups can be found in \cite{Koch}).  When~$S$ contains all~$p$-adic places (i.e.\ in the wildly ramified case), the study of~$G_S$ motivated the development of the Poitou–Tate duality theory. Further study of these wildly ramified pro-$p$ extensions, often guided by analogies with the theory of Riemann surfaces, produced explicit presentations of~$G_S$ in various cases.

In the early 2000s, Labute~\cite{Labute} introduced the notion of \textit{mild} pro-$p$ groups (see  \S \ref{Def Mild} for definition) for odd primes~$p$ to provide examples of tamely ramified~$G_S$ of cohomological dimension~$2$: these pro-$p$ groups are infinite and  FAB (i.e.\ every open subgroup has finite abelianization),  a property that follows from class field theory. This result was later extended to the case~$p=2$ by Labute and Min{\'a}{\v c} in~\cite{Labute-Minac}. Labute's idea has since inspired a range of further developments. Notably, Schmidt~\cite{schmidt2008uber} proved a general result showing that, by suitably enlarging the set~$S$ to allow additional tame ramification, the Galois groups~$G_{S}$ become mild.

\smallskip

Massey products, originally introduced in a geometrical context (see ~\cite{massey1998higher}) as finer group invariants than cohomology algebras, were independantly used by Morishita~\cite{morishita2004milnor} and Vogel~\cite{vogel2004massey} to extract information about the presentations of~$G_S$. Gärtner~\cite{gartner2015higher} further connected these results with mild presentations.
Building on Dwyer's works~\cite{dwyer1975homology}, which established a connection between Massey products and unipotent representations, Min{\'a}{\v c} and Tân conjectured the following two necessary conditions for pro-$p$ groups to be isomorphic to the pro-$p$ absolute Galois group~$G_K$ of a field~$K$.

\vskip 5pt

$(i)$ The Kernel Unipotent Conjecture~\cite[Conjecture~1.3]{minavc2015kernel} predicts that the Zassenhaus filtration of~$G_K$, when~$K$ contains a primitive~$p$-th root of the unity~$\zeta_p$, is given by the kernels of unipotent representations.

\vskip 5pt

$(ii)$ The Massey vanishing conjecture, which is the main focus of this paper, was first formulated in~\cite{minavc2016triple} under the assumption~$\zeta_p \in K^{\times}$, and later extended to a more general form in~\cite{minac2016triple}. The conjecture can be formally stated as follows:

\begin{conjectureNC}[Min{\'a}{\v c}-Tân]
Let~$G_K$ be a pro-$p$ absolute Galois group over a field~$K$.
Then~$G_K$ satisfies the Massey vanishing property.
\end{conjectureNC}

We refer the reader to  \S \ref{MVP} for the relevant definitions.
This conjecture has been the subject of active research over a wide range of base fields (see, for example~\cite{merkurjev2024lectures} for a detailed overview of known results). In the arithmetic context,
 Min{\'a}{\v c} and Tân~\cite[Theorem~$4.3$]{minavc2016triple} showed the conjecture for $G_K$, when $K$ is a local field. Subsequently, Guillot, Min{\'a}{\v c}, Topaz and Wittenberg~\cite{guillot2018four} verified the Massey vanishing property in the case~$n=4$ and~$p=2$, for~$G_K$ when~$K$ is a number field. Finally, Harpaz and Wittenberg~\cite{harpaz2023massey} completely resolved the conjecture for number fields.

\smallskip

Let~$G$ be a pro-$p$ group, and let~$n \geq 3$ be an integer. For an $n$-tuple~$(\chi_1,\dots, \chi_n)$ of homomorphisms from~$G$ to~$\F_p$, if the Massey product~$\langle \chi_1, \dots, \chi_n\rangle$ is defined  (see \S~\ref{MVP} for definition), then the cup products $\chi_j\cup \chi_{j+1}$ vanish for all~$1\leq j\leq n-1$. This motivates the following properties on $G$, known as the strong Massey vanishing property:

\smallskip
\begin{center}
\textit{If~$\chi_j\cup \chi_{j+1}=0$ for all~$1 \leq j \leq n-1$, then the Massey product~$\langle \chi_1,\dots \chi_n\rangle$ vanishes.}     
\end{center}

\smallskip
Min{\'a}{\v c} and Tân~\cite{minavc2017counting} studied this property for pro-$p$ absolute Galois groups. However, its validity turns out to be more delicate. Harpaz and Wittenberg showed that the strong Massey vanishing property does not hold for $G_K$, when $K$ is a number field containing an~$8$-th primitive root of the unity ~\cite[Appendix]{guillot2018four}. Merkurjev and Scavia~\cite{merkurjev2023degenerate} generalized the previous argument for several other fields.

By contrast, the second author, together with Min{\' a}{\v c}, Ramakrishna, and Tân~\cite{maire2024strong} verified the strong Massey vanishing property for $G_K$, when $K$ is a number field not containing the~$p$-th roots of unity, and its tame absolute quotient.
These contrasting results illustrate the increased complexity that the presence of the~$p$-th roots of unity in~$K$ may bring to the study of~$G_K$.

\subsection*{Results}

In this work, we study the strong Massey vanishing property in the context of restricted ramification, and we use it to construct unipotent extensions of number fields with small number of ramified primes. By a unipotent extension, we mean a Galois extension whose Galois group is isomorphic to~$\U_n \coloneq  \U_n(\F_p)$, the group of~$n \times n$ upper triangular unipotent matrices over~$\F_p$.

\smallskip

It has been observed that realizing a finite group as a Galois group comes at the cost of introducing significant ramification. If~$K$ does not contain~$\zeta_p$, the generalization of the Scholz-Reichardt method from~$\Q$ to~$K$ already requires a nontrivial amount of ramification. When~$\zeta_p \in K$, the situation becomes considerably more difficult. For example, for general fields~$K$, no general reasonable upper bound is known for the number of ramified primes needed to realize arbitrary~$2$-groups (see, for instance,~\cite{Schmid}). This again shows that the presence of~$\zeta_p$ in the base field~$K$ affects the complexity of studying~$G_K$.

\smallskip

On the other hand, wild ramification enables the construction of a variety of finite~$p$-groups~$G$ as Galois groups over number fields unramified outside~$p$. Specifically, for a prime~$p$ and a number field~$K$ with~$r_2$ complex places, the Galois group of the maximal pro-$p$ extension of~$K$ unramified outside~$p$ is often a free pro-$p$ group of rank~$r_{2}+1$ (see~$p$-rationality in \S 2.3). By the universal property of free pro-$p$ groups, any finite~$p$-group~$G$ with generator rank at most~$r_{2}+1$ can then be realized as a Galois group over~$K$, unramified outside~$p$. In particular, we can easily demonstrate that any finite~$p$-group~$G$ can be realized as a Galois group~$Gal(L/K')$ over some number field~$K'$, such that the extension~$L/K'$ is unramified outside~$p$.% (see Remark~\ref{rema_end}).

\smallskip

However, using only wild ramification has limitations when the base field is fixed, as the generator rank of such Galois groups is bounded by~$r_{2}+1$. To go beyond this, it is necessary to allow tame ramification.
Inspired by  Wingberg~\cite{MR697311}, Movahhedi~\cite{MR1124802} and, Jaulent and Sauzet~\cite{JaulentSauzet}, we study new situations where pro-$p$ extensions with restricted ramification and splitting condition is a coproduct of free and Demushkin components. These groups check the strong Massey vanishing property, which allows us, in the best situation, to infer unipotent quotients with generator rank~$2r_{2}+2$, which is twice the maximal generator rank allowed in the purely wild case.

\

Let us introduce notations before stating our results. Let~$K$ be a number field of signature~$(r_1,r_2)$. Let~$S$ and~$T$ be finite sets of primes of~$K$ such that~$T$ is disjoint from $S$, and let~$S_p$ denote the set of primes of $K$ lying above $p$.
Define
$$\delta_S \coloneq \sum_{\p \in S \cap S_p}[K_\p:\Q_p], \quad
r^T_S \coloneq \delta_S-(r_1+r_2-1+|T|).$$
Let~$K^T_S$ be the maximal pro-$p$ extension of~$K$ unramified outside~$S$ and totally decomposed at~$T$, and set~$G^T_S \coloneq Gal(K^T_S/K)$ and~$G^{T,ab}_S \coloneq G^T_S/[G^T_S,G^T_S]$. 
For the definition of the strong Massey vanishing property, see Definition \ref{defi-MVP} in \S \ref{MVP}.

\begin{theoremA}\label{theoA}
Suppose that~$G^{T,ab}_S \simeq \Z_p^{r^T_S}$. Then there exist infinitely many sets~$N$ of tame primes of~$K$ with~$|N|=r^T_S$ such that~$G_{S\cup N}^T$ satisfies the strong Massey vanishing property. As a consequence, if~$r^T_S \geq 2$, then there exists a surjective homorphism $$\rho_S: G^T_{S\cup N} \twoheadrightarrow \U_{2r^T_S+1},$$ where~$\U_{2r^T_S+1}$ denotes  the group of upper-unitriangular matrices of size~$2r^T_S+1$ over~$\F_p$.
\end{theoremA}

As initiated in~\cite{maire2024strong} and inspired by~\cite{conti2025liftinggaloisrepresentationskummer}, we can lift the coefficients of the map~$G_{S\cup N}^T \twoheadrightarrow \U_{2r^T_S+1}$ to~$\Z/p^m \coloneq \Z/p^m\Z$, provided that the set~$N$ of primes and the integer~$m$ are suitably chosen. Furthermore, if the matrix is not required to be of maximal size $2r^T_S+1$, one can work over~$\Z/p^m$ for any~$m$.
For instance:

\begin{theoremA}\label{theoB}
Let~$K$ be a number field with~$r_2\geq 2$, and assume the Gras and Leopoldt Conjectures. Fix an integer~$m \geq 1$. Then, for all sufficiently large primes~$p$, there exist infinitely many sets~$N$ of~$r_2$ tame primes such that there exists a surjective homomorphism~$G_{S_p\cup N} \twoheadrightarrow \U_{2r_2+2}(\Z/p^m)$.
\end{theoremA}

Here~$\mathbb{U}_{n+1}(\Z/p^m)$ denotes the group of  upper-triangular unipotent~$(n + 1)\times (n + 1)$-matrices with entries in~$\Z/p^m$. A more general version that incorporates splitting conditions is given in Heuristic~\ref{heu-Wief} in \S \ref{subsection-so}.

\smallskip

Gras's Conjecture is central in this kind of study, and we refer the reader to  Conjecture~\ref{conjGras} in \S~\ref{section_gras} for a more precise formulation.

\

This paper is organised in two parts. The first part focuses on group theoretical results on the strong Massey vanishing property and unipotent representations. The second part studies arithmetic applications and proves our results. For the computations, we have used the program PARI/GP~\cite{PARI2}.

\subsection*{Notations} We fix a prime number~$p$. 

$\bullet$~For a~$\Z_p$-module~$A$, the number~$d_p A$ refers to the dimension of~$A/A^p$ over~$\F_p$, and~$\rk_{\Z_p} A$ denotes the~$\Z_p$-rank of~$A$, which is the dimension over~$\Q_p$ of~$\Q_p \otimes_{\Z_p} A$.

$\bullet$~If~$X$ is a set, we denote by~$|X|$ its cardinality.

$\bullet$~We denote by~$\Gamma$ a finitely presented pro-$p$ group.

$\bullet$~Almost all cohomology groups~$H^i(\Gamma,\F_p)$ have~$\F_p$-coefficients with trivial action so in those cases we simply write
~$H^i(\Gamma)$. %and we define~$h^i(\Gamma) \coloneq \dim_{\F_p}H^i(\Gamma)$.
We denote by~$d_p \Gamma  \coloneq  h^1(\Gamma)$ the generator rank of~$\Gamma$, and by~$h^2(\Gamma)$ its relation rank.

$\bullet$~Set~${\rm Frat}(\Gamma) \coloneq \Gamma^p[\Gamma,\Gamma]$ be the Frattini subgroup of~$\Gamma$. For every finitely generated pro-$p$ group homomorphism
~$\rho\colon \Gamma \to \Gamma'$, we denote by~$\rho^{\rm Frat}\colon \Gamma/\Frat (\Gamma)\to \Gamma'/\Frat(\Gamma')$, the induced homomorphism.

%%%%%%%%%%%%%%%%%%%%%%%%%%%%%%%%%%%%%%%%%%%%
%%%%%%%%%%%%%%%%%%%%%%%%%%%%%%%%%%%%%%%%%%%%%%%%%%%%%

 \section{Group Theory}

\subsection{Massey vanishing property and liftings}

Let~$n \geq 3$ be an integer, and~$\mathbb{U}_{n+1}$ be the group of all upper-triangular unipotent~$(n + 1)\times (n + 1)$-matrices
with entries in~$\F_p$. We denote by~$Z_{n+1}$ the subgroup of~$\bU_{n+1}$ with all off-diagonal entries~$0$ except at position~$(1, n + 1)$: this subgroup is the center of~$\bU_{n+1}$ and is isomorphic to~$\F_p$. We define the quotient group~$\overline{\bU_{n+1}} \coloneq \bU_{n+1}/Z_{n+1}$, which can be seen as the class of matrices where the $(1,n+1)$-entries are formally removed.

 Let~$\psi\colon \bU_{n+1}\to \overline{\bU_{n+1}}$ be the canonical surjection, and define the maps
$$ \varphi  \colon \bU_{n+1} \to \F_p^n, \quad 
M \mapsto (M_{1,2}, \dots, M_{n,n+1}),\quad
\overline{\varphi} \colon \overline{\bU_{n+1}} \to \F_p^n, \quad 
\overline{M} \mapsto (\overline{M}_{1,2}, \dots, \overline{M}_{n,n+1}).$$

For each continuous group homomorphism $\rho : \Gamma \to \U_{n+1}$ and each $1 \leq i < j \leq n+1$, we denote by $\rho_{i,j}$ the $(i,j)$-th coordinate function : 
$$\rho_{i,j} : \Gamma \rightarrow \F_p, \quad g \mapsto \rho(g)_{i,j}.$$
We use similar notation
for homomorphisms~$\overline{\rho} : \Gamma \rightarrow \overline{\bU_{n+1}}$.
Note that~$\rho_{i,i+1}$ (resp.~$\overline{\rho}_{i,i+1}$) is a
group homomorphism for each $1 \leq i \leq n-1$.

\smallskip

By definition, we have the commutative diagram of groups:

$$\xymatrix{
1 \ar[r] &Z_{n+1}  \ar[r] & \bU_{n+1} \ar[r]^\psi \ar@{->>}[rd]_{\varphi} &  \overline{\bU}_{n+1}\ar@{->>}[d]^{\overline{\varphi}}\ar[r] &1  \\
& & & \F_p^n }
$$

\subsubsection{Massey vanishing properties}\label{MVP}
For each $n$-tuple $\chi \coloneq (\chi_1, \dots , \chi_n)$ of elements in~$H^1(\Gamma)$, we denote by $\theta_\chi$ the map :
$$\theta_\chi\colon \Gamma \to \F_p^n, \quad g\mapsto (\chi_1(g), \dots, \chi_n(g)).$$

\begin{defi}
We say that:
\begin{enumerate}
\item[$-$] The Massey product~$\langle \chi_1,\dots ,\chi_n\rangle$ is \textit{defined} if~$\theta_\chi$ lifts to~$\overline{\bU_{n+1}}$, i.e.\ there exists a morphism~$\overline{\rho}_\chi\colon \Gamma \to \overline{\mathbb{U}_{n+1}}$ such that~$\theta_\chi \coloneq \overline{\varphi}\circ \overline{\rho}_\chi$.
\item[$-$] The Massey product~$\langle \chi_1, \dots, \chi_n\rangle$ \textit{vanishes} if~$\theta_\chi$ lifts to~$\bU_{n+1}$, i.e.\ there exists a morphism~$\rho_\chi\colon \Gamma \to \mathbb{U}_{n+1}$ such that~$\theta_\chi=\varphi\circ \rho_\chi$.
\end{enumerate}
\end{defi}

The existence of a homomorphic lift of~$\theta_\chi$ to~$\overline{\bU_{n+1}}$  (which is the definition of the Massey product)
is related to the existence of a subset of~$H^2(\Gamma)$ (see~\cite[Theorem~$2.4$]{dwyer1975homology}), which is called the Massey product and denoted by~$\langle \chi_1,\dots, \chi_n\rangle$. Observe that if a Massey product~$\langle \chi_1,\dots, \chi_n \rangle$ vanishes, then it is necessarily defined. If the Massey product is defined, then an easy cohomological computation shows that $\chi_u\cup \chi_{u+1}=0$ for every~$1\leq u \leq n-1$.

\begin{defi}[Massey vanishing property]\label{defi-MVP}
We say that a pro-$p$ group~$\Gamma$
\begin{enumerate}
\item[$-$] satisfies the Massey vanishing property if every defined Massey product vanishes,
\item[$-$] satisfies the strong Massey vanishing property if for every $n$-tuple $(\chi_1,\dots, \chi_n)$ of elements~ in~$H^1(\Gamma)$ satisfying~$\chi_u\cup \chi_{u+1}=0$ for each~$1 \leq u \leq n-1$, the Massey product~$\langle\chi_1,\dots, \chi_n\rangle$ vanishes.
\end{enumerate}
\end{defi}

%%%%%%%%%%%%%%%%%%%%%%%%%%%
%%%%%%%%%%%%%%%%%%%%%%%%%%%%%%%%%%%%%%%%%%%%%%%%

\subsubsection{Liftings and~$m$-strong Massey vanishing property}

 For every integer~$m\geq 1$, set~$\Z/p^m \coloneq \Z/p^m\Z$. Let
$\mathbb{U}_{n+1}(\Z/p^m)$ be the group of all upper-triangular unipotent~$(n + 1)\times (n + 1)$-matrices
with entries in~$\Z/p^m$. We define the surjective morphism
$$\varphi_m\colon \bU_{n+1}(\Z/p^m)\to \F_p^n; \quad M \mapsto (\overline{M_{1,2}}; \dots; \overline{M_{n-1,n}}),$$
where~$\overline{M_{i,j}}$ is the image of~$M_{i,j}$ modulo~$p^{m-1}(\Z/p^m)$. Note that when $m = 1$, the map $\varphi_1$ coincides with $\varphi$ defined previously.

Building on ideas from~\cite{conti2025liftinggaloisrepresentationskummer}, the second author, together with Min{\'a}{\v c}, Ramakrishna and Tân~\cite{maire2024strong}, also introduced the following group-theoretic property, which they studied in the context of (tame) absolute Galois groups:

\begin{defi}[$m$-strong Massey vanishing property]
We say that a pro-$p$ group~$\Gamma$ satisfies the~$m$-strong Massey vanishing property if, for every $n$-tuple $(\chi_1, \dots ,\chi_n)$ of elements  in~$H^1(\Gamma)^n$ satisfying~$\chi_i\cup \chi_{i+1}=0$ for each~$1 \leq i \leq n-1$, there exists a morphism~$\rho_{m, \chi}\colon \Gamma \to \bU_{n+1}(\Z/p^{m})$  such that the following diagram commutes:

\begin{center}
\begin{tikzcd}

\Gamma \arrow[rr, "\rho_{m,\chi}", dashed] \arrow[rrdd, "\theta_\chi"] &  & \mathbb{U}_{n+1}(\Z/p^m)\arrow[dd, "\varphi_m",two heads]           \\
                                                                                  &  &                                                  \\
                                                                                  &  & \F_p^n
\end{tikzcd}
\end{center}

\justifying
 We also say that~$\rho_{m,\chi}$ is a lifting of~$\theta_\chi$ (over~$\Z/p^m$).
\end{defi}

If~$m=1$, this is the usual strong Massey vanishing property.

We now use the~$m$-strong Massey vanishing property to lift unipotent surjections with coefficients in~$\F_p$ to unipotent surjections with coefficients in~$\Z/p^m$. 

\begin{lemm}\label{frat uni}
    For every integer~$m\geq 1$, the morphism induced by~$\varphi_m$ denoted
    $$\varphi_{m}^{\rm Frat}\colon \bU_{n+1}(\Z/p^m)/{\rm Frat}(\bU_{n+1}(\Z/p^m))\to  \F_p^n$$
    is an isomorphism.
\end{lemm}

\begin{proof}
Clearly~${\rm Frat}(\F_p^n)=0$.
    Let us denote by~${\rm I}_{n+1}$ the identity matrix and by~$\epsilon_{i,j}$ the elementary~$(n+1)\times (n+1)$-matrix which is equal to one in~$(i,j)$, and zero everywhere else. An easy computation shows that:
 $$\lbrack \bU_{n+1}(\Z/p^m),\bU_{n+1}(\Z/p^m)\rbrack =\langle {\rm I}_{n+1}+  a  \epsilon_{i,j}; \quad a \in \Z/p^m \text{ and } j-i\geq 2 \rangle.$$
Furthermore,
$$\bU_{n+1}(\Z/p^m)^p \cdot [\bU_{n+1}(\Z/p^m),\bU_{n+1}(\Z/p^m)] =P +[\bU_{n+1}(\Z/p^m),\bU_{n+1}(\Z/p^m)],$$
where $P$ is the additive subgroup generated by $\{ pb\epsilon_{i,j} ; \quad b\in \Z/p^m, j-i=1\}$.

    This implies that the kernel of~$\varphi_m$ is exactly~${\rm Frat}(\bU_{n+1}(\Z/p^m))$. Thus~$\varphi_m^{\rm{Frat}}$ is an isomorphism.
\end{proof}

As an application, we infer:

\begin{prop}\label{surj and strong}
Assume that~$\Gamma$ satisfies the~$m$-strong Massey vanishing property. Suppose moreover  that there exists  an $n$-tuple ~$\chi \coloneq (\chi_1, \dots, \chi_n)$ in~$H^1(\Gamma)^n$ satisfying the following conditions:
\begin{enumerate}
\item[$-$] the map~$\theta_\chi\colon \Gamma \to \F_p^n$ is surjective,
\item[$-$] for every~$1\leq u\leq n-1$, we have~$\chi_u\cup \chi_{u+1}=0$.
\end{enumerate}
Then there exists a surjective homomorphism~$\rho_{m,\chi}\colon \Gamma \to \bU_{n+1}(\Z/p^m)$ lifting~$\theta_\chi$. Moreover, all such liftings are surjective.
\end{prop}

\begin{proof}
    From the~$m$-strong Massey vanishing property, there exists~$\rho_{m,\chi}\colon \Gamma \to \bU_{n+1}(\Z/p^m)$ which lifts~$\theta_\chi$ in~$\Z/p^m$. Let us recall from Lemma~\ref{frat uni} that we have the isomorphism $$\varphi_m^{\rm Frat}\colon \bU_{n+1}(\Z/p^m)/{\rm Frat}(\bU_{n+1}(\Z/p^m))\simeq \F_p^n.$$
Thus we infer the commutative diagram:

\centering
\begin{tikzcd}
\Gamma/{\rm Frat}(\Gamma)\simeq \F_p^d \arrow[rrdd, "\theta_\chi^{\rm Frat}"', two heads] \arrow[rr, "\rho_{m,\chi}^{\rm Frat}", dashed] &  & \mathbb{U}_{n+1}(\Z/p^m)/{\rm Frat}(\mathbb{U}_{n+1}(\Z/p^m)) \arrow[dd, "\varphi_m^{\rm Frat}"] \\
                                                                                                                                         &  &                                                                                                        \\
                                                                                                                                         &  & \F_p^n
\end{tikzcd}

Since~$\theta_\chi$ is surjective and~$\varphi_m^{\rm Frat}$ is an isomorphism, then the map~$\rho_{m,\chi}^{\rm Frat}$ is surjective. By Burnside Lemma, we conclude that the map~$\rho_{m,\chi}$ is also surjective.
\end{proof}

%%%%%%%%%%%%%%%%%%%%%%%%%%%%%%%
%%%%%%%%%%%%%%%%%%%%%%%%%%%%%%%%%%%

\subsection{Mildness}\label{Def Mild}
Let us consider a finitely presented pro-$p$ group~$\Gamma$ with a presentation (not necessarily minimal):
$$(P_\Gamma) \coloneq \langle x_1,\dots x_d \mid  l_1,\dots ,l_r\rangle.$$
Alternatively, we have a presentation~$1\to R \to F \to \Gamma \to 1$, where~$F$ is pro-$p$ free on~$d$ generators and~$R$ is the closed normal subgroup of~$F$ generated by the~$l_i$'s.

We denote by~$E(\Gamma)$ the completed group ring of~$\Gamma$ over~$\F_p$ filtered by~$E_n(\Gamma)$, the~$n$-th power of the augmentation ideal.

The Magnus isomorphism (see~\cite[Chapitre II,~$3.1.4$ and Appendice~$A.3$]{LAZ}) provides an isomorphism~$\phi$ of filtered algebras between~$E(F)$ and~$E \coloneq \F_p\langle \langle X_1,\dots, X_d \rangle \rangle$, the algebra of noncommutative series in~$X_1, \dots X_d$ over~$\F_p$, where every~$X_i$ is assigned degree~$1$. The isomorphism $\phi$ is characterized by $\phi(x_i) \coloneq X_i+1$.

Let us choose an order~$\succ_X$ on~$\{X_1,\dots, X_d\}$,  that we extend to an order on monomials on~$E$. To fix the ideas, we take the order~$X_d \succ_X X_{d-1}\succ_X \dots \succ_X X_1$. This is always possible after relabeling the $X_i$'s. 

We define~$\widehat{l_i}$ as the leading monomial of the series~$\phi(l_i)-1$.

\begin{defi}[Mild groups]
    In this paper, we say that the presentation~$(P_\Gamma)$ is (quadratic) mild (for an order~$\succ_X$) if for every~$i$, we have~$\widehat{l_i} \coloneq X_{i_2}X_{i_1}$ (with~$i_2>i_1$), and for every~$i,j$, we have~$X_{i_2}\neq X_{j_1}$.
\end{defi}

This is (a special case of) the notion used by
 Labute in~\cite[$\S 1$]{Labute} to produce examples of pro-$p$ group~$G_{S}$ of cohomological dimension~$2$.

\begin{prop}
If~$(P_\Gamma)$ is mild, then the presentation is minimal and~$\Gamma$ has cohomological dimension~$2$.
\end{prop}

\begin{proof}
See~\cite[Theorem~$5.1$]{Labute}.
%and ~\cite[Proposition~$2.3$]{polishchuk2005quadratic}. The second is discussed in~\cite[Proposition 1]{hamza2023extensions}.
\end{proof}

\begin{rema}[Koszulity and Mildness]
    Following our definition, using~\cite[Theorem~$5.1$]{Labute} and~\cite[Proposition~$1$]{hamza2023extensions}, we can easily show that if the group~$\Gamma$ admits a mild presentation, then the algebra~$H^\bullet(\Gamma)$ is Koszul. Positselski~\cite{positselski2014galois} conjectured that pro-$p$ absolute Galois groups of fields containing the $p$-th roots of the unity has Koszul cohomology ring, which is a stronger property than the Bloch-Kato conjecture. This conjecture was investigated by Min{\'a}{\v c} and his collaborators in~\cite{minac2021koszul} and~\cite{minavc2022mild}. 
\end{rema}

\

\subsection{The property~$(\PP_m)$ and the class~$\D_m$}  \label{property P}   Set~$m\geq 1$.
We define the property~$(\PP_m)$ that we study for the rest of the paper:
\begin{defi}
We say that a group~$\Gamma$ satisfies the property~$(\PP_m)$ if:
\begin{itemize}
\item[$(i)$] the group~$\Gamma$ checks the~$m$-strong Massey vanishing property,
\item[$(ii)$] the group~$\Gamma$ admits a mild quadratic presentation or is free,
\item[$(iii)$] every open subgroup of~$\Gamma$ checks~$(i)$ and~$(ii)$.
\end{itemize}
\end{defi}

We now define the class~$\D_m$ of pro-$p$ absolute Galois groups~$\Gamma$ of local fields with characteristic of residue fields different from~$p$, which either:

$-$ do not contain the~$p$-th roots of the unity, so~$\Gamma \simeq \Z_p$~\cite[Theorem~$10.1$]{Koch},

$-$ contain the~$p^m$-th roots of the unity, so~$\Gamma$ is Demushkin of rank~$2$~\cite[Theorem~$10.2$]{Koch}.

\smallskip
Unipotent representations on this class were studied by Conti, Demarche and Florence in~\cite{conti2025liftinggaloisrepresentationskummer}. We easily observe that if~$m\geq m'$ then~$\D_m$ is a subclass of~$\D_{m'}$. %We also notice that if~$\Gamma$ is in~$\D_m$, then every open subgroup~$H$ of~$\Gamma$ is in~$\D_m$ (see for example \cite[Proposition 7.5.9, Chapter VII]{NSW}).

\begin{prop}\label{DmPPm}
    The class~$\D_m$ checks the property~$(\PP_m)$.
\end{prop}

\begin{proof}
    Take~$\Gamma$ in~$\D_m$. We check~$(i)$,~$(ii)$ and~$(iii)$. If~$\Gamma\simeq \Z_p$, this is clear. So we assume that~$\Gamma\neq \Z_p$.

$(i)$ From~\cite[Proposition $4.1$]{minavc2017counting}, we observe that every group in~$\D_m$ checks the strong Massey vanishing property.
%\footnote{\textcolor{red}{I think in the paper of Quadrelli, he mentions a paper of Pal and Szabo}} 
Let~$(\chi_1,\dots \chi_n)$ be an ~$n$-tuple in~$H^1(\Gamma)^n$ satisfying~$\chi_u\cup \chi_{u+1}=0$. By the strong Massey vanishing property, there exists~$\rho_\chi\colon \Gamma \to \bU_{n+1}$ lifting~$\theta_\chi$. Using~\cite[Corollary]{conti2025liftinggaloisrepresentationskummer}, we infer a morphism~$\rho_{m,\chi}\colon \Gamma \to \bU_{n+1}(\Z/p^m)$ which lifts~$\rho_\chi$, thus lifts~$\theta_\chi$. Consequently~$\Gamma$ satisfies the~$m$-strong Massey vanishing property.

   $(ii)$ Using~\cite[Theorem~$10.2$]{Koch}, we observe that~$\Gamma$ admits a presentation with two generators~$x_2$,~$x_1$ and one relation~$l_1$ which checks~$\widehat{l_1}=X_2X_1$, for the order~$X_2>X_1$, so it is mild.

   $(iii)$ 
   Every open subgroup~$H$ of~$\Gamma$ is in~$\D_m$ (see for example \cite[Proposition 7.5.9, Chapter VII]{NSW}), so satisfies~$(i)$ and~$(ii)$.
\end{proof}

\subsection{Coproducts and the class~$\overline{\D_m}$}

We denote by~$\coprod$ the coproduct in the category of pro-$p$ groups: this is the pro-$p$ completion of the abstract free product in groups. For further details, we refer the reader to~\cite[Chapter~$9$]{ZAL} and~\cite[Chapter IV]{NSW}.

Let~$\{G_1, \dots, G_k\}$ be a family of finitely generated pro-$p$ groups, and set~$G \coloneq \coprod_{i=1}^k G_i$. We observe that we have natural injective morphisms~$\iota_j\colon G_j\to G$ for every~$1\leq j\leq k$, and we infer:

\begin{prop}[Universal property]\label{cohomology coprod}
For every pro-$p$ group~$M$ and family of maps~$\{ \, \rho_j \colon G_i \to M\}_{1 \leq j \leq k}$, there exists a unique map:
$$\rho \colon G \to M,$$
that we denote by~$\rho \coloneq \coprod_{i=1}^k \rho_i$, such that for each~$j$ the following diagram commutes:

\centering{ 
\begin{tikzcd}
G_j \arrow[rd, "\rho_j"] \arrow[rr, "\iota_j", hook] &   & G \arrow[ld, "\rho", dashed] \\
                                                     & M &
\end{tikzcd}}
\end{prop}

\begin{proof}
See~\cite[Proposition~$9.1.2$]{ZAL}.
\end{proof}

Let us observe that we have a map 
$${\rm Res}_1\colon H^1(G)\to \bigoplus_{i=1}^k H^1(G_i), \quad \chi_u \mapsto (\chi_u \circ \iota_1,\dots, \chi_u \circ \iota_k).$$
By Proposition~\ref{cohomology coprod} this map is bijective, and we identify every element~$\chi$ in~$H^1(G_i)$ with  an element~$\widetilde{\chi}$ in~$H^1(G)$ verifying~$\widetilde{\chi}(\iota_i(g))=\chi(g)$ for every~$g\in  G_i$, and~$\widetilde{\chi}(\iota_j(G_j))=0$ for~$j\neq i$.

\begin{prop}\label{cohomology of coprod}
    For every integer~$n\geq 1$, we have an isomorphism
    $${\rm Res}_n\colon H^n(G)\simeq \bigoplus_{i=1}^k H^n(G_i)$$
    Furthermore for every~ pair $i\neq j$, the image of the following map is trivial:
    $$\cup\colon H^1(G_i)\times H^1(G_j)\to H^2(G), \quad (\chi_1, \chi_2)\mapsto \widetilde{\chi_1}\cup \widetilde{\chi_2}.$$ 
\end{prop}

\begin{proof}
The isomorphism comes from~\cite[Theorem~$(4.1.4)$]{NSW} and the previous discussion.
The computation on the coproduct is also well known, 
but let us propose an alternative proof using unipotent representations, inspired by~\cite[Lemma~$4.7$]{minavc2017counting}. The proof follows easily by induction from the case~$G \coloneq G_1\coprod G_2$.  Take~$\chi_1$ in~$H^1(G_1)$ and~$\chi_2$ in~$H^1(G_2)$. We define morphisms~$\rho_1\colon G_1\to \bU_3$ and~$\rho_2\colon G_2 \to \bU_3$ by:

\[
\rho_1 \colon g_1 \mapsto 
\begin{bmatrix}
1 & \chi_1(g_1) & 0 \\
0 & 1 & 0 \\
0 & 0 & 1
\end{bmatrix}
\quad \text{and} \quad
\rho_2 \colon g_2 \mapsto 
\begin{bmatrix}
1 & 0 & 0 \\
0 & 1 & \chi_2(g_2) \\
0 & 0 & 1
\end{bmatrix}
\]

Then using the universal property, we infer a morphism~$\rho \coloneq (\rho_1\coprod \rho_2)\colon G\to \bU_3$ which satisfies~$\rho_{1,2}=\widetilde{\chi_1}$ and~$\rho_{2,3}=\widetilde{\chi_2}$. Furthermore, an easy computation implies~$\widetilde{\chi_1}\cup \widetilde{\chi_2}=0$. 
\end{proof}

\subsubsection{Mildness}
Now let us assume that, for each~$i = 1, \dots, k$, the group~$G_i$ is either free or admits a mild presentation~$(P_{G_i})$. In this case, the universal property implies the following:

\begin{coro}\label{mild coprod}
    The group~$G \coloneq \coprod_{i=1}^k G_i$ either admits a mild presentation, or is pro-$p$ free.
\end{coro}

\begin{proof} If every~$G_i\simeq \Z_p$, then~$G$ is pro-$p$ free. Now to simplify the notations, we assume that for every~$i$, the pro-$p$ group~$G_i$ is not free (the mixed case is very similar).

    Thus for every~$i$ the group~$G_i$ admits a mild presentation:
    $$(P_{G_i}) \coloneq \langle x_{i,1}, \dots, x_{i,d_i} \mid l_{i,1}, \dots , l_{i,r_i}\rangle,$$
    and we have~$\widehat{l_{i,j}} \coloneq X_{i,j_2}X_{i,j_1}$, where~$X_{i,j_2}> X_{i,j_1}$ and~$j_2>j_1$.

    By Proposition~\ref{cohomology coprod}, the group~$G$ admits a presentation:
    $$(P_G) \coloneq \langle x_{1,1}, \dots, x_{1,d_1}, \dots , x_{k,d_k} \mid l_{1,1}, \dots, l_{1,r_1}, \dots, l_{k,r_k}\rangle.$$
    Considering the order:~$X_{k,d_k}\succ_X X_{k,d_k-1}\succ_X\dots\succ_X X_{k,1}\succ_X\dots \succ_X X_{1,1}$, we conclude that the presentation~$(P_G)$ is mild.
\end{proof}

\subsubsection{Open subgroups and the class $\overline{\D_m}$}
We define~$\overline{\D_m}$ as the closure of~$\D_m$ with respect to finite coproducts. Concretely, the pro-$p$ groups in~$\overline{\D_m}$ are described by pro-$p$ groups~$G \coloneq \coprod_{i=1}^k G_i$, where~$\{G_1, \dots, G_k\}$ is a family of pro-$p$ groups in~$\D_m$. A profinite version of the Kurosh subgroup Theorem \cite[Theorem~$(4.2.1)$]{NSW} allows us to show the following result:

\begin{coro}\label{virt} Suppose that the pro-$p$ groups~$G_1,\cdots, G_k$  are in~$\D_m$. Set~$G=\coprod_{i=1}^k G_i$.
    Then every open subgroup~$H$ of~$G$ is a coproduct of groups in~$\D_m$.
\end{coro}

\begin{proof}
    By \cite[Theorem~$(4.2.1)$]{NSW}, the group~$H$ is a coproduct of groups that are either isomorphic to an open subgroup of some $G_i$, or to a free pro-$p$ group. We conclude by applying Proposition \ref{DmPPm}.    
\end{proof}

\subsubsection{The property~$(\PP_m)$ and the class $\overline{\D_m}$}
We show here that the class $\overline{\D_m}$ satisfies the property $(\PP_m)$

\begin{theo}\label{stab coprod}
Assume that~$G$ is a pro-$p$ group in~$\overline{\D_m}$, then~$G$ satisfies~$(\PP_m)$.
\end{theo}

\begin{proof}
If $G$ is free, the proof is easy. Let us consider a family~$G_1,\dots, G_k$ such that~$G \coloneq \coprod_{i=1}^kG_i$, and assume that every~$G_i$ is mild (the proof of the mixed case is very similar). 
We observe from Corollary~\ref{mild coprod} that~$G$ is mild, so checks~$(ii)$.
Furthermore, from Corollary~\ref{virt}, every open subgroup of~$G$ is in~$\overline{\D_m}$. To conclude, we only need to show that~$G$ satisfies the~$m$-strong Massey vanishing property. This is a well-known result, but let us give a proof. 

From Proposition~\ref{cohomology of coprod}, we observe that
$$H^1(G)\simeq \bigoplus_{i=1}^k H^1(G_i), \ \text{and } \widetilde{\chi_{i,a}} \cup \widetilde{\chi_{j,b}}=0 \text{ for } i\neq j$$
for every~$\chi_{i,a} \in H^1(G_i) $ and $\chi_{j,b} \in H^1(G_j)$.

Let us consider~$(\chi_1, \dots, \chi_n)$ an $n$-tuple of elements in ~$H^1(G)$ such that~$\chi_u\cup \chi_{u+1}=0$. We construct a morphism~$\rho\colon G \to \mathbb{U}_{n+1}(\Z/p^m)$ such that~$\rho_{u,u+1} \equiv \chi_u \pmod{p}$ for each $1 \leq u \leq n$. For this purpose, we write
$$\chi_u \coloneq \sum_{j=1}^k \widetilde{\chi_{j,u}}, \ \text{ where } \chi_{j,u} \coloneq \chi_u\circ \iota_j \text{ is in } H^1(G_j).$$
As~${\rm Res}_2$ is an isomorphism that satisfies~${\rm Res}_2(a\cup b)={\rm Res}_1(a)\cup {\rm Res}_1(b)$, we infer that:
$$\chi_u\cup \chi_{u+1} \coloneq \sum_{j=1}^k \widetilde{\chi_{j,u}}\cup \widetilde{\chi_{j,u+1}}=0 \implies \widetilde{\chi_{j,u}}\cup \widetilde{\chi_{j,u+1}}=0, \text{ for every } j.$$

Since~$G_j$ satisfies the~$m$-strong Massey vanishing property, we can construct~$\rho_j\colon G_j \to \mathbb{U}_{n+1}(\Z/p^m)$ such that~$(\rho_j)_{u,u+1} \equiv \chi_{j,u} \pmod{p}$. 

By the universal property, we infer a map~$\rho\colon G \to \mathbb{U}_{n+1}(\Z/p^m)$, which satisfies~$\rho_{u,u+1} \equiv \sum_j \widetilde{{\chi_{j,u}}} \coloneq \chi_u \pmod{p}$.
\end{proof}

Let us give consequences on unipotent quotients of~$G$.

\begin{prop}\label{unipotent quotients}
    Assume that~$G$ is in~$\overline{\D_m}\setminus \D_m$, i.e.,\ ~$G$ is a coproduct of at least two factors. Then the group~$\mathbb{U}_{n+1}(\Z/p^m)$ is a quotient of~$G$ if and only if~$n\leq h^1(G)$.
\end{prop}

\begin{proof}
    If~$\bU_{n+1}(\Z/p^m)$ is a quotient of~$G$, then we infer a surjection~$G/{\rm Frat}(G)\simeq \F_p^{h^1(G)}$ to~$\bU_{n+1}(\Z/p^m)/{\rm Frat}(\bU_{n+1}(\Z/p^m))\simeq \F_p^n$. Thus, we obtain~$h^1(G) \geq n$.

    Conversely, we assume that~$h^1(G)\geq n$, and let us write~$G \coloneq G_1\coprod \dots \coprod G_k$ where~$k\geq 2$ and each~$G_i$ lies in~$\D_m$. By Theorem~\ref{stab coprod}, the pro-$p$ group~$G$ satisfies the~$m$-strong Massey vanishing property. Using Proposition~\ref{cohomology of coprod}, we can construct an $n$-tuple~$(\chi_1,\dots, \chi_n)$ of characters of~$G$ such that~$\chi_u\cup \chi_{u+1}=0$ for every $1 \leq u \leq n-1$.

    To simplify the discussion, let us assume that none of the groups~$G_i$ is isomorphic to~$\Z_p$. For each $1 \leq i \leq k$ we choose~$\{\chi_{i,1}, \chi_{i,2}\}$ a basis of~$H^1(G_i)$, and we define~$\chi_i \coloneq  \chi_{i,1}$ if~$i\leq k$ and~$\chi_i \coloneq \chi_{i-k,2}$ if~$i>k$.  This family is well defined since~$h^1(G) \coloneq 2k \geq n$, and the associated map~$\theta_\chi \colon G \to \F_p^n$ is surjective. We conclude using Proposition~\ref{surj and strong}.
\end{proof}

\begin{rema}\label{unipotent quotients 2}
Assume that~$G \in \overline{\D_m}\setminus \D_m$, and is given in the form:
$$G \coloneq \bigg ( \coprod_{i=1}^{k_1}G_i \bigg ) \coprod \bigg ( \coprod_{j=1}^{k_2}H_j \bigg ),$$
where~each $G_i$ is a Demushkin group of rank~$2$ in~$\D_m$, and each~$H_j$ is isomorphic to 
$\Z_p$. In particular, we have~$h^1(G)=2k_1+k_2$.
The pro-$p$ group~$G$ admits a free quotient of rank~$k_1+k_2$. Using Lemma~\ref{frat uni}, we infer that for every integer~$l$, the pro-$p$ group~$G$ admits~$\bU_{n+1}(\Z/p^l)$ as a quotient, whenever~$n\leq k_1+k_2$. However, thanks to the property~$(\PP_m)$, we can go beyond this bound: more precisely, we can construct a surjection from $G$ onto~$\bU_{n+1}(\Z/p^m)$ for any~$n\leq 2k_1+k_2$.
\end{rema}

%%%%%%%%%%%%%%%%%%%%%%%%%%%%%%%%%%%%%%%%%%%
%%%%%%%%%%%%%%%%%%%%%%%%%%%%%%%%%%%%%%%%%%%%
%%%%%%%%%%%%%%%%%%%%%%%%%%%%%%%%%%%%%%%%%%%%%%%%%%%%
%%%%%%%%%%%%%%%%%%%%%%%%%%%%%%%%%%%%%%%%%%%%%%%%%%%%%%%
%%%%%%%%%%%%%%%%%%%%%%%%%%%%%%%%%%%%%%%%%%%%%%%%%%%%%%%%%%%%%%%

\section{Arithmetic applications}

\subsection{Notations}\label{arithmetic notations}
Let~$K$ be a number field.
Denote by

\begin{itemize}
\item[$\bullet$]~$(r_1,r_2) \coloneq (r_{1,K},r_{2,K})$ the signature of~$K$,
\item[$\bullet$]~$S_p$ the set of~$p$-adic places of~$K$,
\item[$\bullet$]~$S$ a finite set of places of~$K$; set~$S_p' \coloneq S\cap S_p$,
\item[$\bullet$]~$K_v$ the completion of~$K$ at each place~$v$ of~$K$, and~$U_v$ the group of units of~$K_v$,
\item[$\bullet$]~$\G_v$ the Galois group of the maximal pro-$p$ extension of~$K_v$;~$\I_v$ its inertia subgroup, and~$\Ff_v = \G_v/\I_v$,
\item[$\bullet$]~$\delta_S \coloneq \underset{v\in S_p'}{\sum}[K_v:\Q_p]$, so that~$\delta_S=\delta_{S_p'}$,
\item[$\bullet$]~A finite prime~$\q$ of~$K$ is called tame if~$N(\q) \equiv 1 \pmod{p}$, and more generally,~$m$-tame if~$N(\q) \equiv 1 \pmod{p^m}$ for some integer~$m \geq 1$,
\item[$\bullet$]~For a set~$N=\{\q_1,\cdots,\q_n\}$ of tame primes, we write~$m_N \coloneq min \{ v_p(N(\q)-1), \q \in N\}$, where~$v_p$ is the discrete~$p$-adic valuation on~$\Z$,
\item[$\bullet$]~For each place~$v$, let~$\displaystyle{\Uu_v \coloneq  \lim_{\stackrel{\longleftarrow}{n}} U_v/U_v^{p^n}}$ be the pro-$p$ completion of~$U_v$. Then define~$\displaystyle{\Uu_{S} \coloneq \prod_{v \in S}\Uu_v}$,
\item[$\bullet$]~$T$ a finite set of places of~$K$,  disjoint  from~$S$;
    set~$E^T \coloneq E_K^T$ the \textit{pro-$p$ completion} of the  group of~$T$-units of~$K$,
\item[$\bullet$]~$\varphi_S^T : E^T  \to \Uu_S$   the diagonal embedding of~$E^T$ into~$\Uu_S$,
\item[$\bullet$]~$K_S^T/K$  the maximal pro-$p$ extension of~$K$ unramified outside~$S$ and totally decomposed at~$T$;~$G_S^T \coloneq G_{K,S}^T \coloneq Gal(K_S^T/K)$, and~$G_S \coloneq G_S^\emptyset$,

\item[$\bullet$]~$K_S^{T,ab}$ is the maximal abelian extension of~$K$ in~$K_S^T$; set~$G_S^{T,ab} \coloneq Gal(K_S^{T,ab}/K)$,

\item[$\bullet$]~$K^{T,p,el}_S$ is the maximal elementary abelian extension of~$K$ in~$K^T_S$; set~$(G^T_S)^{p,el} \coloneq Gal(K^{T,p,el}_S/K)$,
\item[$\bullet$]~$\T_S^T$ the~$\Z_p$-torsion part of~$G_S^{T,ab}$,
\item[$\bullet$]~$r_S^T \coloneq  \rk_{\Z_p} G_S^{T,ab}$,~$r_S \coloneq r_S^\emptyset$,
\end{itemize}

In our work, infinite places play a limited role in arguments. We focus on finite places and distinguish them by notation according to their roles:~$\p$ denotes a~$p$-adic place,~$\q$ a non-$p$-adic (tame) place where ramification may occur (see Remark~\ref{normcondition}), and~$\l$ a place at which splitting conditions are imposed.

\subsection{Classical results}

Most of the results in this section are well-known; see for example~\cite[\S 1]{lim2024analyticity}, \cite[Chapter X]{NSW}. Since Shafarevich and Koch, we know that~$G_S^T$ is finitely presented as in the following theorem.
\begin{theo} Suppose that~$S$ is not empty. Then, we have
$$1-h^1(G_S^T)+h^2(G_S^T) \leq r_1+r_2+|T| - \delta_S.$$
\end{theo}

When~$S'_p=\emptyset$, the pro-$p$ group~$G_S^T$ is FAB. More generally, we have

\begin{prop}\label{Z_p-rank}
One has~$r_S^T= \rk_{\Z_p} \left( {\rm coker} (\varphi_S^T) \right)$. Thus if~$\varphi_S^T$ is injective then $$r_S^T =\delta_{S} -(r_1+r_2-1+|T|).$$
Conversely, if~$T$ is disjoint from~$S$ and we have the previous equality, then~$\varphi_S^T$ is injective.
\end{prop}

As a consequence, we have (see. \cite[Lemma 1.3]{lim2024analyticity}):

\begin{coro}  \label{coro_bounded_H2} Suppose that~$S \neq \emptyset$. Then
$\rk_{\Z_p} H_{2}(G_{S}^{T}, \Z_p) \leq  \rk_{\Z_p} \left(\ker (\varphi_{S}^{T})\right)$,
where~$H_2$ denotes the second group homology.
\end{coro}

By duality, we have an isomorphism~$H^2(G, \Q/\Z) \cong H_2(G, \Z_p)$ for a pro-$p$ group~$G$. Corollary~\ref{coro_bounded_H2} allows us to find many instances where the following lemma is particularly useful.

 \begin{lemm}\label{lemm_free} Let~$\psi : \Gamma \twoheadrightarrow G$ be a surjective morphism of pro-$p$ groups. Suppose moreover that~$H^2(G,\Q/\Z)=0$. Then~$\psi$ is an isomorphism if and only if~$\psi$ induces an isomorphism between~$\Gamma^{ab}$ and~$G^{ab}$.
 \end{lemm}

 \begin{proof}
 See~\cite[Lemma~$2$]{MR1124802}.
 \end{proof}

In particular, we have the following practical criterion for determining when~$G^T_S$ is free.

\begin{prop}\label{prop_free}
A pro-$p$ group~$G$ is free if and only if~$G^{ab}$ is torsion-free and~$H^2(G,\Q/\Z)=0$. In particular, if~$\ker(\varphi_{S}^T)=1$ and~$\T_{S}^T=1$, then~$G^T_S$ is free pro-$p$. Furthermore, we have~$h^1(G_S^T) = 1+\delta_S-(r_1+r_2+|T|)$.
\end{prop}

\begin{proof}
Consider a minimal presentation
$$1\to R \to F \to G \to 1.$$
Since~$G^{ab}$ is torsion free, we have~$G^{ab}\simeq F^{ab}$. Therefore, by Lemma~\ref{lemm_free}, we conclude that~$F\simeq G$. Now, let~$K$,~$S$, and~$T$ be as in the statement of Proposition~\ref{prop_free}. In this setting, the vanishing of~$H^2(G^T_S, \Q/\Z)$ follows from Corollary~\ref{coro_bounded_H2}. Lastly, the formula for~$h^1(G^T_S)$ follows from the assumption~$\T^T_S=1$, together with the Burnside Basis lemma and Proposition~\ref{Z_p-rank}.
\end{proof}

\begin{rema}
A non-$p$-adic place is not ramified in a free pro-$p$ extension. Hence, if~$G^T_S$ is free, then we have~$K^T_S=K^T_{S'}$.
\end{rema}

\begin{rema}(see~\cite[\S 10]{Koch})\label{normcondition}
A finite non-$p$-adic prime~$\q$ ramifies in a pro-$p$ extension if and only if its norm~$N(\q)$ in~$\NN$ satisfies~$N(\q) \equiv 1 \pmod{p}$. It is well known that the following conditions are equivalent:
\begin{itemize}
\item~$\q$ is~$m$-tame;
\item~$\G_{\q}$ belongs to the class~$\D_m$;
\item~$\q$ splits completely in~$K(\zeta_{p^m})/K$.
\end{itemize}
\end{rema}

\begin{rema}
When~$S=S_p$ and~$T=\emptyset$, the conditions~$\ker(\varphi_{S_p})=1$ and~$\T_{S_p}=1$ together are equivalent to the freeness of~$G_{S_p}$, by Euler-Poincar\'{e} characteristic formula (see~\cite[Corollary 8.7.5]{NSW}).
\end{rema}

\subsection{The coproduct decomposition of~$G_S$} \label{section_gras}

In this subsection, we explain that, in many cases, the Galois group~$G_S$ belongs to the class~$\overline{\D_m}$ for some integer~$m \geq 1$.

One important guiding principle in the study of~$G_S$ has been to use presentations of~$\G_v$ at the places~$v \in S$, which is well-known, as explained in \S~\ref{property P}.

\medskip

For each place~$v$, we have a natural restriction map~$\G_v \to G_S$ corresponding to a fixed embedding~$k_S \hookrightarrow \overline{k_v}$. When~$v \not\in S$, this map factors through the quotient~$\G_v \to \G_v/\I_v \cong \Ff_v$. Using Proposition~\ref{cohomology coprod}, for each subset~$S_0$ of~$S$ we obtain a canonical morphism $$\psi_{S,S_0}: \coprod_{v \in S_0} \G_v \longrightarrow G_S.$$
By the Burnside Basis theorem, the map~$\psi_{S,S_0}$ is surjective if and only if~$G^{S_0}_{S\setminus S_0}$ is trivial.

If~$\psi_{S,S_0}$ is not surjective, we choose a finite set~$W$ of primes of~$K$, disjoint from~$S$, such that~$G^{S_0 \cup W} _{S \setminus S_0}$ is trivial. Then, we obtain a surjective map
$$ \psi_{S,S_0,W} : \bigg ( \coprod_{v \in S_0} \G_v \bigg ) \coprod \bigg ( \coprod_{w \in W} \Ff_w \bigg  ) \twoheadrightarrow G_S. $$

Inspired by the theory of Riemann surfaces, several studies have investigated the problem of finding conditions under which the map~$\psi_{S,S_0,W}$ is an isomorphism (see~\cite[Chapter X, \S 5 and \S 9]{NSW}). This situation serves as a starting point for our study of~$G^T_S$.

To illustrate the relevance of such cases, suppose that the map~$\psi_{S,S_0,W}$ is an isomorphism, that~$S_0$ contains at least one tame prime, and that each local Galois group~$\G_v$ with~$v \in S_0 \cap S_p$ is free. Then~$G_S$ belongs to the class~$\overline{\D_m}$, where~$m = m_{S'_0}$ for the maximal subset~$S'_0 \subseteq S_0$ consisting of tame primes. By Theorem~\ref{stab coprod}, it follows that~$G_S$ satisfies the property~$(\PP_m)$.

\medskip

The notion of~$p$-rationality naturally emerges in this context and will play a central role in what follows. Throughout, we make the additional assumption that~$K$ is totally imaginary when~$p=2$. Then, the cohomological dimension of~$G_{S_p}$ is always less than or equal~$2$. Following~\cite[Definition 1]{MR1124802}, we say that a number field~$K$ is~$p$-rational if~$G_{S_p}$ is free. While this condition may seem technical, the case~$\Q(\zeta_p)$ with~$p$ regular provides a classical example (see~\cite[\S 4]{Sha}). In recent years, there has been growing interest in $p$-rationality for multiquadratic fields. Additionally, we recall the following fundamental conjecture due to Gras.

\begin{conj}[Gras~\cite{Gras-CJM}, Conjecture~$8.11$]\label{conjGras} Given a number field~$K$, we have~$\T_{S_p}=1$ for~$p\gg 0$ and so, from Leopoldt Conjecture, the field~$K$ is~$p$-rational for~$p\gg 0$.
\end{conj}

The role of~$p$-rationality is highlighted in the following theorem.

\begin{theo}[Satz~$3.1$ of~\cite{MR697311}, Théorème~$2$ of~\cite{MR1124802}]\label{mainthmGS}
Assume that~$S$ contains~$S_p$. The following conditions are equivalent:

$(i)$ The field~$K$ is~$p$-rational and the Frobenius automorphisms at the places of~$S \setminus S_p$ form a basis of~$(G_{S_p})^{p,el}$.

$(ii)$ The map $$\psi_{S,S \setminus S_p,\emptyset} : \coprod_{\q \in S \setminus S_p} \, \G_{\q} \longrightarrow G_S$$ is an isomorphism.
\end{theo}

As a result, once a number field is~$p$-rational, we can find many~sets $N$ of tame primes such that~$G_{S_p \cup N}$ belongs to~$\overline{\D}$. We explore this idea further with~$G^T_S$ in the next section.

\begin{rema}
In~\cite{Wingberg89}, Wingberg established the necessary and sufficient condition on~$(K,S)$ for the free pro-$p$ product decomposition of~$G_S$ in a more general setting. Theorem 6 of~\cite{Wingberg89}, whose proof does not require~$\zeta_p \in K$, recovers Theorem~\ref{mainthmGS}. The free pro-$p$ product decomposition of~$G_{S_p}$ into a free pro-$p$ group and Demushkin groups of rank greater than~$2$, corresponding to the absolute pro-$p$ Galois group of a~$p$-adic local field containing~$\zeta_p$~\cite[Theorem 7.5.14]{NSW}, is in principle possible. This was studied further in~\cite{JaulentSauzet} through explicit examples.
\end{rema}

%%%%%%%%%%%%%%%%%%%%%%%%%%%%%%%%%%%%%%%%%%%%%%%%%%%%%%%%%%%%%%%%%%%%%%%
%%%%%%%%%%%%%%%%%%%%%%%%%%%%%%%%%%%%%%%%%%%%%%%%%%%%%%%%%%%%%%%%%%%%%%%

\subsection{On the freeness of~$G^T_S$}\label{sec-free}
We now aim to extend Theorem~\ref{mainthmGS}, originally stated for~$G_S$, to the setting of~$G^T_S$, while also allowing the set~$S$ not to contain~$S_p$. Generally, the difficulty in studying (pro-$p$) Galois groups comes from the absence of a general theory of Galois cohomology. We use Corollary~\ref{coro_bounded_H2} and  Lemma~\ref{lemm_free}, which offer a more heuristic and direct approach, thereby avoiding the need for deep cohomological machinery. In this respect, our approach differs from that of~\cite{MR1124802, Wingberg89}, which used the Poitou–Tate duality. In particular, we proceed under the assumption that~$G^T_S$ is free.

Nevertheless, proving the freeness of~$G^T_S$ is a highly non-trivial and transcendental problem. Rather than pursuing a direct proof, we provide a heuristic argument which suggests that~$G^T_S$ is free in many cases. As shown in Proposition~\ref{prop_free}, two properties would be essential: the injectivity of~$\varphi_S^T$ and the triviality of~$\T_S^T$.

The relationship between the injectivity of~$\varphi^T_S$ and the Schanuel Conjecture is well-known. For instance, we have:

\begin{prop}\label{schanuel}
Let~$K$ be a Galois extension over an imaginary quadratic field~$k$ with Galois group~$G$. Let~$p$ be a prime that splits in~$k$, and fix a prime~$\p$ of~$k$ above~$p$. Let~$S_{\p}$ denote the set of primes of~$K$ lying above~$\p$, and let~$S$ be a finite set of primes of~$K$ containing~$S_{\p}$. If~$G$ is abelian, then~$\varphi^T_S$ is injective for~$T=\{\l\}$, where~$\l$ is any non-$p$-adic prime of~$K$. Moreover, assuming Schanuel's Conjecture, the injectivity of~$\varphi^{\{\l\}}_S$ holds for arbitrary Galois extension~$K/k$.
\end{prop}

\begin{proof} See~\cite[\S 3]{CM_JNT} and~\cite[Chapter III, Corollary 3.6.5]{Gras}
\end{proof}

In~\cite{Gras-CJM}, Gras introduced a heuristic argument for Conjecture~\ref{conjGras}. Building on his approach, we anticipate the following heuristic:

\begin{heuristic}\label{heuristicGST}
Let~$K$ be a number field, and let~$T$ be a fixed finite set of primes of~$K$ such that for each~$\l \in T$, the completion~$K_{\l}$ is equal to~$\Q_\ell$ for~$\ell$ lying below~$\l$. Assume that~$|T| {\leq} r_2$. Then, under Conjecture~\ref{conjGras}, the number of~$p$ for which~$\ker(\varphi^T_{S_p}) \neq 1$ or~$\T^T_{S_p} \neq 1$ is expected to be finite.
\end{heuristic}

In fact, if~$G_S$ is known to be free, it becomes straightforward to obtain a sufficient condition for the freeness of $G^T_S$. To streamline our discussion, we introduce the following definition.

\begin{defi} Let~$S$ and~$T$ be two finite disjoint sets of places of~$K$. A set of places~$W$, disjoint with~$S \cup T$, is said to be~$(S,T)$-primitive if the Frobenius automorphisms in~$(G^T_S)^{p,el}$ at~$W$ are linearly independent over~$\F_p$. If the Frobenius automorphisms in~$(G^T_S)^{p,el}$ at~$W$ are a basis over~$\F_p$, we say that~$W$ is maximal~$(S,T)$-primitive. In the special case~$T=\emptyset$, we simply call such a set~$S$-primitive.
\end{defi}

We note that in~\cite{MR1124802}, the term  \textit{primitivity} was used to refer to~$S_p$-primitivity in the context of~$p$-rational fields. We have the following proposition

\begin{prop}\label{S-primitive cons} Suppose that~$G_S$ is a free pro-$p$ group. For any~$S$-primitive set~$T$, the group~$G^T_S$ is also free pro-$p$. Moreover, if~$\ker (\varphi_S)=1$, then we also have~$\ker( \varphi^T_S)=1$.
\end{prop}

\begin{proof}
Let~$R$ be the normal subgroup of~$G_S$ generated by the Frobenius automorphisms at~$T$. By the freeness of~$G_S$, we have the exact sequence
$$ 0 \to H^1(G^T_S) \to H^1(G_S) \to H^1(R)^{G^T_S} \to H^2(G^T_S) \to 0. $$
From the~$S$-primitivity of~$T$, we deduce that~$h^1(G^T_S)=h^1(G_S)-|T|$.  Moreover, since~$R$ is the closed normal subgroup generated by the Frobenius automorphisms at~$T$, we have~$d_p H^1(R)^{G^T_S}\leq  |T|$. Thus, equality must hold:~$d_p H^1(R)^{G^T_S}=|T|$, which implies that~$h^2(G^T_S)=0$, and hence~$G^T_S$ is free. The second claim follows from the chain of equalities
$$ |T| = d_p G_{S} - d_p G^T_{S} = r_{S}- r^T_{S} =
\rk_{\Z_p}\ker (\varphi_{S}) - \rk_{\Z_p}\ker (\varphi^T_{S}) + |T|, $$
which shows that~$\ker (\varphi_S^T)=1$, since by hypothesis we have~$\ker (\varphi_S)=1$. 
\end{proof}

As a consequence of the Chebotarev density theorem, once the freeness of~$G_S$ is established, one can find many sets~$T$ such that~$G^T_S$ is also free. In contrast, the heuristic we employ takes a different perspective: it fixes~$K$ and~$T$, and studies the freeness of~$G^T_{S_p}$ as~$p$ varies.

\begin{proof}[Supporting argument for Heuristic~\ref{heuristicGST}]
By the Gras conjecture, we can assume without loss of generality that~$G_{S_p}$ is free of rank~$r_2+1$. According to the Chebotarev density theorem, for each \textit{fixed}~$p$, the Frobenius automorphism of primes~$\l$ in~$(G_{S_p})^{p,el}$ is equidistributed as~$\l$ \textit{varies}. Since the Dirichlet density of the set of primes~$\l$ with~$K_{\l} \neq \Q_\ell$ is zero, the equidistribution still holds when restricting to primes with~$K_{\l}=\Q_\ell$.
 Hence, for a fixed~$\l \in T$, it is reasonable to heuristically expect the Frobenius automorphism in~$G^{p,\mathrm{el}}_{S_p}$ at~$\l$ to be equidistributed as~$p$ varies. The group~$G^T_{S_p}$ is free of rank~$r_2+1-|T|$ unless the Frobenius automorphisms in~$(G_{S_p})^{p,el}$ associated to the places in~$T$ are linearly dependent. Hence, the probability~$\P(p,K,T)$ that~$G^T_{S_p}$ is not free of rank~$r_2+1-|T|$ is equal to
$$ 1 - \prod^{|T|}_{i=1} \bigg ( 1 - \frac{1}{p^{r_2-i+2}} \bigg ). $$
One can check that the infinite sum~$\sum_p \P(p,K,T)$ is bounded above. The claim about finiteness then follows from the Borel-Cantelli lemma (see the beginning of~\cite[\S 4.1]{Gras-CJM}).
\end{proof}

\begin{rema}
\begin{itemize}
\item[(i)] While our heuristic is formulated under the assumption of the Gras Conjecture, it is worth noting that the argument in~\cite{Gras-CJM} can be used to heuristically recover both the conjecture on~$G_{S_p}$ and the same expectation on~$G^T_{S_p}$, simultaneously and from the same reasoning. Unlike our heuristic,~\cite{Gras-CJM} uses the equidistribution of the generalized Fermat quotient (ex.~\cite[\S 4.2.1.(ii)]{Gras-CJM}) and the heuristic on \textit{the existence of a binomial probability law} (\cite[\S 4.4]{GrasFQ},~\cite[\S 7]{Gras-CJM}), as observed through numerical experiments.

\item[(ii)] Our assumption that~$K_{\l}=\Q_\ell$ was made mainly for simplicity. For example, if~$K$ is a Galois number field and the decomposition subgroup~$D_{\l}$ of~$G \coloneq G(K/\Q)$ at~$\l$ is nontrivial, then the Frobenius automorphism at~$\l$ lies in the sub-$\F_p[Gal(K/\Q)]$-module of~$(G_{S_p})^{p,el}$ on which~$D_{\l}$ acts trivially. The presence of such a non-trivial Galois action can create an obstruction to deducing the finiteness of the set of primes as predicted in Heuristic~\ref{heuristicGST}. Nevertheless, the overall heuristic suggests that such primes, though possibly infinite, are still very rare.
\end{itemize}
\end{rema}

\begin{rema}
We may attempt to apply the idea from the proof of Proposition~\ref{S-primitive cons} to obtain a free quotient~$G^T_{S_p}$ from a non-free pro-$p$ group~$G_{S_p}$. However, this is not straightforward. Controlling the Frobenius elements in~$G_S^{ab}$ at~$T$ is subtle, and for~$G^{T,ab}_{S_p}$ to be torsion-free, the map~$\varphi^T_{S_p}$ must fail to be injective. As a consequence, under the Leopoldt Conjecture, for any non-$p$-adic prime~$\l$, the group~$G^{\{\l\}}_{S_p}$ is not free unless~$G_{S_p}$ is.
\end{rema}

\subsection{Proof of Theorem~\ref{theoA}} \label{sec-mainresult}

%%%%%%%%%%%%%%%%%
We now prove Theorem~\ref{theoA}, beginning with the following result: 

\begin{prop}\label{abelianization}
 Assume that~$\varphi^T_S$ is injective. Let~$N=\{\q_1,\cdots, \q_s\}$ be a finite set of tame primes. Then $$Gal(K_{S\cup N}^{T,ab}/K_S^{T,ab}) \simeq \Z/p^{n_1} \times \cdots \times \Z/p^{n_r}, $$ where~$n_i =v_p (N(\q_i)-1)$. If moreover~$\T_S^T=1$, then we have
$$G_{S\cup N}^{T,ab}\simeq \Z_p^{r_S^T} \times \Z/p^{n_1} \times \cdots \times \Z/p^{n_r}\cdot$$
\end{prop}

\begin{proof}
From the commutative diagram
\begin{equation*}
\begin{tikzcd}
0 \arrow[r] & E^T \arrow[r] \arrow[d, no head, equal] & \Uu_{S \cup N} \arrow[r] \arrow[d, two heads] & (G^T_{S \cup N})^{ab} \arrow[r] \arrow[d, two heads] & \Cl_T \arrow[r] \arrow[d, no head, equal] & 0 \\
0 \arrow[r] & E^T \arrow[r]                                & \Uu_{S} \arrow[r]                             & (G^T_{S})^{ab} \arrow[r]                             & \Cl_T \arrow[r]           & 0
\end{tikzcd}
\end{equation*}
and the Snake lemma, we infer that~$Gal(K_{S\cup N}^{T,ab}/K_S^{T,ab})$ is isomorphic to the kernel of the map~$\Uu_{S \cup N}/\varphi_{S \cup N}(E^T) \to \Uu_{S}/\varphi_{S}(E^T)$. Here,~$\Cl_T$ denotes the quotient of the~$p$-class group of~$K$ by the subgroup generated the ideal classes of the primes in~$T$. We deduce the claim by applying the Snake lemma once more to
\begin{equation*}
\begin{tikzcd}
0 \arrow[r] & E^T \arrow[d, no head, equal] \arrow[r] & \Uu_{S \cup N} \arrow[d, two heads] \arrow[r] & \Uu_{S \cup N}/\varphi_{S \cup N}(E^T) \arrow[d, two heads] \arrow[r] & 0 \\
0 \arrow[r] & E^T \arrow[r]                                & \Uu_S \arrow[r]                               & \Uu_S/\varphi_S(E^T) \arrow[r]                               & 0
\end{tikzcd}
\end{equation*}
The second claim follows from the first one by the~$\Z_p$-freeness of~$(G^T_S)^{ab}$.
\end{proof}

The following result corresponds to Theorem~\ref{theoA} in the introduction. For additional explanation, see also Remark~\ref{rema_arithmetic_main}.

\begin{theo}[Theorem~\ref{theoA}] \label{theo_arithmetic_main} Let~$S$,~$T$,~$N\subset M$,  be four sets of primes of~$K$ such that
\begin{enumerate} 
\item[$(i)$] the map~$\varphi_S^T$ is injective,
 \item[$(ii)$] the torsion part~$\T_S^T$ of~$(G_S^T)^{ab}$  is trival,
 \item[$(iii)$] the primes~$\q$ in~$N$ are tame,
 \item[$(iv)$]~$M$ is maximal~$(S,T)$-primitive.
 
\end{enumerate}
 Then the natural map~$\psi_{S\cup N, N, M\setminus N}^T : \left(\coprod_{\q \in N} \G_{\q} \right) \coprod  \left( \coprod_{\l \in M\backslash N } \Ff_\l \right)  \longrightarrow G_{S\cup N}^T$ is an isomorphism. In particular, the pro-$p$ group~$G_{S\cup N}^T$ satisfies~$(\PP_{m_N})$, and consequently the~$m_N$-strong Massey vanishing property, where~$m_N \coloneq min \{v_p(N(\q)-1), \q \in N\}$.
\end{theo}

\begin{proof}
By Proposition~\ref{prop_free} and the assumptions~$(i)$ and~$(ii)$, the group~$G^T_S$ is free of rank~$r^T_S$. Consider the natural map
$$\psi^T_{S\cup N,N,M\setminus N} : \left(\coprod_{\q \in N} \G_{\q} \right) \coprod  \left( \coprod_{\l \in M\backslash N } \Ff_\l \right)\longrightarrow G^T_{S\cup N}.$$

By assumption~$(iv)$, the map $\psi^T_{S\cup N,N,M\setminus N}$ is surjective. Hence, by Proposition \ref{abelianization}, it induces an epimorphism on the abelianizations between isomorphic finitely generated $\Z_p$-modules. By the structure theorem for finitely generated modules over $\Z_p$, it follows that the map on the abelianizations is an isomorphism.

 Moreover, assumption~$(i)$ and Corollary~\ref{coro_bounded_H2} together imply that~$H^2(G^T_S, \Q/\Z)=0$. Therefore by Lemma~\ref{lemm_free}, the map~$\psi^T_{S\cup N,N,M\setminus N}$ is an isomorphism. As a consequence, the group~$G^T_{S \cup N}$ belongs to the class~$\overline{\D_{m_N}}$ and, by Theorem~\ref{stab coprod}, it satisfies the property~$(\PP_{m_N})$.
\end{proof}

\begin{rema}\label{rema_arithmetic_main}
    The hypothesis~$(G_S^T)^{ab}\simeq \Z_p^{r_S^T}$ in Theorem~\ref{theoA} immediately yields ~$(ii)$ in Theorem~\ref{theo_arithmetic_main}. Furthermore, if~$T$ is disjoint from~$S$, then by Proposition~\ref{Z_p-rank}, condition~$(i)$ also holds. Applying the arguments from the previous proof and using the Chebotarev density theorem, we can find infinitely many sets~$N$ of size~$r_S^T$ such that $$\psi_{S\cup N, N, \emptyset}^T : \left(\coprod_{\q \in N} \G_{\q} \right)  \longrightarrow G_{S\cup N}^T$$
    is an isomorphism. By using Proposition \ref{unipotent quotients} (and Remark \ref{unipotent quotients 2} with $k_2=0$), this establishes Theorem~\ref{theoA} as stated in the introduction.
\end{rema}

\begin{exem}
Consider the multiquadratic field~$K=\Q(\sqrt{-1}, \sqrt{2}, \sqrt{7}, \sqrt{19})$, and take~$p=3$. Let~$N$ be a set consisting of one prime of~$K$ lying over each of the rational primes~$19, 31, 199$, and let~$T$ be a set consisting of one prime of~$K$ lying over each of~$53$ and~$89$. Then, we have
$$ G_{S_3 \cup N}^T \cong \coprod_{v \in N} \G_v \coprod \mathcal{F} $$
where~$\mathcal{F}$ is the free pro-$3$ group of rank~$4$.
\end{exem}

\begin{coro}\label{coro-uni-m=1}Let~$K$ be a number field, and~$S \subset S_p$ such that:
\begin{enumerate} \item[$(i)$] the map~$\varphi_S$ is injective,
 \item[$(ii)$]~$\T_S=1$.
 \end{enumerate}
Assume moreover that~$r_S\geq 2$. Then there exists infinitely  many sets~$N$ of tame primes, with~$|N|= r_S$, such that $$G_{S\cup N} \twoheadrightarrow \U_{2r_S+1}\cdot$$ As a consequence, there exists a~$\U_{2r_S+1}$-extension of~$K$ unramified outside~$S \cup N$.
\end{coro}

\begin{rema}
    Corollary~\ref{coro-uni-m=1} is a direct consequence of Theorem~\ref{theo_arithmetic_main} and Proposition~\ref{unipotent quotients}. Observe that if~$r_S=1$ then, for~$N=\{\q\}$, the group~$G_{S\cup N}$  is a Demushkin group of rank~$2$. Thus~$\bU_3$ is not a quotient of~$G_{S\cup N}$.
\end{rema}

%%%%%%%%%%%%%%%%%%%%%%%
%%%%%%%%%%%%%%%%%%%%%%%%%%%%%%%%%%

\subsection{Lifting Massey Products to~$\Z/p^m$}
In this subsection, we enlarge~$S$ to construct surjective unipotent representations~$G^T_{S\cup N} \to \U_{n}(\Z/p^m)$ for~$m > 1$.
Theorem~\ref{theo_arithmetic_main} provides situations where~$G_{S\cup N}^T$ lies in~$\overline{\D_{m_N}}$ and satisfies~$(\PP_{m_N})$. By Proposition~\ref{unipotent quotients}, this leads to a surjective homomorphism~$G_{S \cup N}^T \to \U_{2r^T_S+1}$. However, in certain cases, the conditions~$(iii)$ and~$(iv)$ of Theorem~\ref{theo_arithmetic_main} may be incompatible, creating an obstruction to constructing unipotent representations with both rank~$n=2r^T_S+1$ and large~$m$.

\subsubsection{Obstruction to full rank with large coefficients}

Let~$\nu$ be the largest integer~$\geq 0$ such that~$K(\zeta_{p^\nu})=K(\zeta_p)$. Let~$K_{1,p}$ be the~$\F_p$-extension of~$K$ contained in~$K(\zeta_{p^{\nu+1}})$. For simplicity of notation, we write~$K_1$ for~$K_{1,p}$ throughout, except in the supporting argument of  Heuristic~\ref{heu-Wief}. We note that~$K_{1}(\zeta_p)=K_{1}(\zeta_{p^{\nu+1}})$.

\begin{prop}\label{primetop}
Let~$K$ be a number field and let~$S$ and~$T$ be finite sets of primes of~$K$ such that~$\ker(\varphi^T_S)=1$ and~$\T^T_S=1$. There exist infinitely many sets~$N$ of tame primes of size~$r^T_S$ such that
$$ G^T_{S \cup N} \cong \coprod_{\q \in N} \G_{\q}, $$
where each~$\G_{\q}$ is a Demushkin group in~$\D_{\nu}$. As a consequence, if~$r^T_S \geq 2$, there exists a surjective homomorphism
$$ G^T_{S \cup N} \twoheadrightarrow \U_{2r^T_S+1}(\Z/p^{\nu}). $$
\end{prop}

\begin{proof}
By Chebotarev density theorem, we can find a finite set~$N$ of primes whose Frobenius at~$Gal(K^{T,p,el}_S(\zeta_p)/K)$ form a basis of the subgroup~$Gal(K^{T,p,el}_S(\zeta_p)/K(\zeta_p))$. Via the isomorphism 
$$Gal(K^{T,p,el}_S(\zeta_p)/K(\zeta_p)) \cong (G^T_S)^{p,el},$$
this implies that~$N$ is~$(S,T)$-primitive. Moreover, since the Frobenius automorphism at each~$\q \in N$ fixes~$K(\zeta_p) = K(\zeta_{p^{\nu}})$, it follows from Remark~\ref{normcondition} that each~$\q$ is~$\nu$-tame. Therefore, by Theorem~\ref{theo_arithmetic_main}, the group~$G^T_{S \cup N}$ satisfies the property~$(\PP_{\nu})$. The conclusion then follows from Proposition~\ref{unipotent quotients}.
\end{proof}

\begin{prop}
Let~$K$ be a number field, and let~$S$ and~$T$ be finite sets of primes of~$K$ such that~$\ker(\varphi^T_S)=1$ and~$\T^T_S=1$. Let~$m > \nu$ be an integer. 
\begin{enumerate}
\item[$(1)$] If~$K_{1}$ is not contained in~$K^{T,p,el}_S$, then there exist infinitely many sets~$N$ of tame primes such that~$G^T_{S \cup N}$ is isomorphic to the coproduct of~$r^T_S$ Demushkin groups, each belonging to the class~$\D_m$. As a consequence, if~$r^T_S \geq 2$, there is a surjective homomorphism~$G^T_{S \cup N} \to \U_{2r^T_S+1}(\Z/p^m)$.

\item[$(2)$] If~$K_{1} \subseteq K^{T,p,el}_S$, then no such set~$N$ of size~$r^T_S$ exists. However, there exist sets~$N_0$ with~$|N_0| = r^T_S - 1$ such that, for any~$M \supseteq N_0$, the group~$G^T_{S \cup M}$ is a coproduct of the form
$$ \bigg( \coprod_{\q \in N_0} \G_{\q} \bigg) \coprod G',$$
where each~$\G_{\q}$ for~$\q \in N_0$ belongs to~$\D_m$, and~$G'$ is either isomorphic to~$\Z_p$ or a Demushkin group in the class~$\D_{\nu}$ but not in~$\D_{\nu+1}$. As a consequence, if~$r^T_S \geq 2$, there is a surjective homomorphism~$G^T_{S \cup N'} \to \U_{2r^T_S}(\Z/p^m)$.
\end{enumerate}
\end{prop}

\begin{proof}
If the intersection~$K(\zeta_{p^m}) \cap K^T_S$ is strictly larger than~$K$, then it must coincide with~$K_1$. Therefore, the first claim follows directly from the proof of Proposition~\ref{primetop}.

Assume now that~$K_1 \subseteq K^T_S$. In this case, we can construct a set~$N_0$ of size~$r^T_S-1$ satisfying the desired property, by using the same argument as in Proposition~\ref{primetop} to the isomorphism
$$ Gal(K^{T,p,el}_S(\zeta_{p^m})/K(\zeta_{p^m})) \cong Gal(K^{T,p,el}_S/K_1). $$
Let~$N_0 \cup \{\q'\}$ be a maximal~$(S,T)$-primitive set. If~$\q'$ is not tame, then~$G^T_{S \cup M}=G^T_{S \cup N_0}$ by Remark~\ref{normcondition}. If~$\q'$ is tame, then it is not~$m$-tame. Otherwise, the Frobenius at~$\q'$ would fix~$K(\zeta_{p^m}) \cap K^{T,p,el}_S=K_1$, which contradicts~$|N_0|=\dim_p Gal(K^{T,p,el}_S/K_1)$. Hence, in this case, the pro-$p$ group~$G^T_{S \cup M}$ is a coproduct of~$\coprod_{\q \in N_0} \G_{\q} \in \overline{\D_{m}}$ and~$\G_{\q'} \in \D_{\nu}$. In both cases,~$\coprod_{\q \in N_0} \G_{\q} \coprod \Z_p$ is a quotient of~$G^T_{S \cup M}$, and the claim on the unipotent representation follows.
\end{proof}

\begin{rema}
As shown in the proof of the above proposition, the inclusion~$K_{1} \subseteq K^T_S$ can be viewed as an obstruction to constructing a set~$N$ for which~$G^T_{S \cup N}$ becomes the coproduct of the maximal number of Demushkin groups in~$\D_m$ for~$m \geq \nu$.
\end{rema}

As a complement to Corollary~\ref{coro-uni-m=1}, we now state the following result, which corresponds to Theorem~\ref{theoB} in the introduction. Note that~$K_1 \subset K_{S_p}$.

\begin{coro}[Theorem \ref{theoB}] \label{coro_prational_pm}
Let~$K$ be a number field with~$r_2 \geq 2$, and assume the Gras and Leopoldt Conjectures. Then there exists a constant~$p_0$ such that for all primes~$p > p_0$, and for any integer~$m \geq 1$, there exists a Galois extension~$L/K$ with Galois group
$$Gal(L/K) \cong \U_{2r_2+2}(\Z/p^m),$$ which is unramified outside~$S_p \cup N$, where~$N$ is a set of~$r_2$ tame primes.

More precisely, for each such~$p$ and~$m$, there exist infinitely many sets~$N$ of~$r_2$ tame primes such that there exists a surjective group homomorphism
$$G_{S_p \cup N} \twoheadrightarrow \U_{2r_2+2}(\Z/p^m).$$
\end{coro}

\begin{rema} 
When $K$ does not contain the $p$-th roots of the unity, the theorem of Scholz and Reichardt demonstrates the existence of a Galois extension~$L/K$ with Galois group isomorphic to~$\U_{n+1}(\Z/p^m)$ which is unramified outside a set~$N_0$ of tame primes of size $$|N_0| = (d_p Cl_K + d_p E_K)+n+(nm-1)(nm-2),$$ where~$Cl_K$ is the class group of~$K$. This shows that constructing unipotent Galois extensions using only tame ramification requires a significantly large amount of ramification.
\end{rema}

\subsubsection{Splitting and Obstructions}\label{subsection-so}

As discussed earlier, the obstruction to constructing a unipotent representation~$G^T_{S \cup N} \to \U_{2r^T_S+1}(\Z/p^m)$ for~$m \geq \nu$, lies entirely in the relationship between~$K_{1}$ and~$K^T_S$. We now continue the discussion, with particular focus on the role of~$T$.

If there is a prime~$\l \in T$ which does not split in~$K_{1}$, then~$K_{1} \not\subseteq K^T_S$. We can always find this situation by enlarging~$T$ by one element.

\begin{prop}
Let~$K$,~$S$, and~$T$ be such that~$\ker(\varphi^T_S)=1$ and~$\T^T_S=1$. Then there exists a Chebotarev density set of primes~$\l \not\in T \cup S$ such that~$\{ \l\}$ is~$(S,T)$-primitive and~$K_{1} \not\subset K^{T \cup \{\l\}}_S$.    
\end{prop}

\begin{proof}
It suffices to take~$\l$ that is inert in~$K_{1}$.
\end{proof}

In the same spirit as \S~\ref{sec-free}, we study the inclusion~$K_{1,p} \subset K^T_{S_p}$ for fixed~$K$ and~$T$, for varying prime numbers ~$p$. For a fixed integer~$a$, a prime~$p$ is called a Wieferich prime to the base~$a$ if $$a^{p-1} \equiv 1 \pmod{p^2}.$$
According to the Gras’s heuristic involving a binomial probability law, the number of Wieferich primes to a fixed base~$a$ is expected to be finite (\cite[Th\'{e}or\`{e}me 4.9]{GrasFQ}). This perspective is closely connected to Gras Conjecture and motivates the following heuristic.

\begin{heuristic}\label{heu-Wief}
Let~$K$ be a number field, and let~$T \neq \emptyset$ be a fixed set of primes as in Heuristic~\ref{heuristicGST}. Then, it is expected that there exists an integer~$p_0 \in \NN$ such that for every prime~$p > p_0$ and every integer~$m \geq 1$, there exist infinitely many sets~$N$ of tame primes for which the group~$G^T_{S_p \cup N}$ is the coproduct of~$r^T_S$ Demushkin groups in the class~$\D_m$.

As a consequence, there exist Galois extensions of~$K$ with Galois group isomorphic to~$\U_{2r^T_S+1}(\Z/p^m)$, unramified outside~$S_p \cup N$ and totally decomposed at~$T$.
\end{heuristic}

\begin{proof}[Supporting argument]
According to Heuristic~\ref{heuristicGST}, it is expected that~$\ker(\varphi^T_{S_p}) = 1$ and~$\T^T_{S_p} = 1$ for all but finitely many primes~$p$. Let~$\l \in T$ be fixed, and let~$l$ be the rational prime below~$\l$. For all but finitely many primes~$p$, the field~$K_{1,p}$ is the compositum of~$K$ and~$\Q_{1,p}$. The prime~$\l$ splits in~$K_{1,p}$ if and only if~$\Q_{1,p} \subseteq K_{\l}$, which in turn holds if and only if~$l$ splits in~$\Q_{1,p}$ {for $p > [K:\Q]$}. This is equivalent to~$p$ being a Wieferich prime to the base~$l$. By~\cite[Th\'{e}or\`{e}me 4.9]{GrasFQ}, the number of such primes~$p$ is expected to be finite.
\end{proof}

We conclude this paper with a numerical example that illustrates our results. 

\begin{exem}
Let~$K = \Q(i, \sqrt{2}, \sqrt{7})$, which is the splitting field of the polynomial
$$
f(x) = x^8 - 32x^6 + 344x^4 - 512x^2 + 1936.
$$
Fix a root~$\theta$ of~$f(x)$, and consider the following prime ideals in~$\O_K$:
$$
\begin{aligned}
&\q_1 = (163, \theta^2 + 10), \quad \q_2 = (37, \theta^2 - 6\theta + 7), \\
&\q_3 = (2341, \theta^2 + 306\theta - 551), \quad \q_4 = (73, \theta^2 - 10\theta + 18), \\
&\l = \left(241, \frac{-241 \theta^7}{5280} + \frac{10363 \theta^5}{5280} - \frac{39283 \theta^3}{1320} + \theta^2 + \frac{212821 \theta}{1320} - 82 \right).
\end{aligned}
$$
Let~$N = \{ \q_1, \q_2, \q_3, \q_4 \}$ and~$T =\{ \l \}$. Then the union~$N \cup T$ is~$S_3$-primitive, and the Galois group~$G^T_{S_3 \cup N}$ is isomorphic to the coproduct~$\coprod_{1 \le i \le 4} \mathcal{G}_{\mathfrak{q}_i}$, and hence belongs to the class~$\overline{\mathcal{D}_2}$. As a consequence, there exists a~$\U_9(\Z/9)$-extension of~$K$ which is unramified outside~$S_3 \cup N$ and totally splitting at~$\l$.
    
\end{exem}

%%%%%%%%%%%%%%%%%%%%%%%%%%%%%%%%%%%%%%%%%%%%%%%%%%%%%%%%%%%%%%%%%%%%%%%
%%%%%%%%%%%%%%%%%%%%%%%%%%%%%%%%%%%%%%%%%%%%%%%%%%%%%%%%%%%%%%%%%%%%%%%
%%%%%%%%%%%%%%%%%%%%%%%%%%%%%%%%%%%%%%%%%%%%%%%%%%%%%%%%%%%%%%%%%%%%%%%
%%%%%%%%%%%%%%%%%%%%%%%%%%%%%%%%%%%%%%%%%%%%%%%%%%%%%%%%%%%%%%%%%%%%%%%

\bibliography{bibactbib3}
\bibliographystyle{plain}
\end{document}